\documentclass[11pt]{amsart}
\usepackage{amsmath,amsfonts,amssymb,amsthm}
\usepackage{latexsym,bm,graphicx}
\usepackage{mathrsfs}
\usepackage{color}

\title[Sharp spectral gap and Li--Yau's estimate]{Sharp Spectral Gap and Li--Yau's Estimate on Alexandrov Spaces}
\author{Zhongmin Qian}
\address{Mathematical Institute\\University of Oxford\\ 24--29 St Giles, OX1 3LB, Oxford, United Kingdom\\ E-mail address: qianz@maths.ox.ac.uk}
\author{Hui-Chun Zhang}
\address{Department of Mathematics\\  Sun Yat-sen University\\ Guangzhou 510275\\ E-mail address: zhhuich@mail2.sysu.edu.cn}
\author{Xi-Ping Zhu}
\address{Department of Mathematics\\  Sun Yat-sen University\\ Guangzhou 510275\\ E-mail address: stszxp@mail.sysu.edu.cn}

\newtheorem{thm}{Theorem}[section]
\newtheorem{prop}[thm]{Proposition}
\newtheorem{lem}[thm]{Lemma}

\newtheorem{cor}[thm]{Corollary}

\theoremstyle{definition}
\theoremstyle{remark}

\newtheorem{rem}[thm]{Remark}

\numberwithin{equation}{section}

\newcommand{\ls}{\leqslant}
\newcommand{\gs}{\geqslant}
\newcommand{\wa}{\widetilde\angle}

\newcommand{\ip}[2]{\left<{#1},{#2}\right>}
\newcommand{\rv}{{\rm vol}}

\newcommand{\R}{\mathbb{R}}
\newcommand{\M}{\mathbb{M}}

\begin{document}

\maketitle
\begin{abstract}In the previous work \cite{zz10-3}, the second and third authors established a Bochner type formula on Alexandrov spaces.
 The purpose of this paper is to give some applications of the Bochner type formula.
 Firstly, we extend the sharp  lower bound  estimates of spectral gap, due to  Chen--Wang \cite{cw94,cw97} and  Bakry--Qian \cite{bq00},
  from smooth Riemannian manifolds to Alexandrov spaces. As an application,  we get an Obata type theorem for Alexandrov spaces. Secondly,
  we obtain (sharp) Li--Yau's estimate for positve solutions of heat equations on Alexandrov spaces.
\end{abstract}
\section{Introduction}
Let $n\gs2$ and  $M$  a compact $n$-dimensional Alexandrov space without boundary. It is well known that the first non-zero eigenvalue of
 the (canonical) Laplacian is given by:
$$\lambda_1(M):=\inf\Big\{\frac{\int_M|\nabla f|^2d\rv}{\int_Mf^2d\rv}:\ f\in Lip(M)\backslash\{0\}\ {\rm and}\ \int_Mfd\rv=0\Big\},$$
where $Lip(M)$ is the set of Lipschitz functions on $M$.

When $M$ is a smooth compact Riemannian manifold, the study of the lower bound estimate of first non-zero eigenvalue $\lambda_1(M)$ has a long history,
 see for example Lichnerowicz \cite{l58}, Cheeger \cite{c70}, Li--Yau \cite{ly79}, and so on. For an overview the reader is referred to the introduction
 of \cite{bq00}, \cite{b04,l04} and Chapter 3 in book \cite{sy94},  and references therein.

Let $M^n$ be a compact $n$-dimensional Riemannian manifold without boundary. Lichnerowicz's estimate asserts that $\lambda_1(M^n)\gs n$ if Ricci curvature
 of the manifold $M^n$ is bounded below by $n-1$. Later Obata \cite{ob62} proved that the equality holds if and only if the manifold $M^n$ is isometric to
  $\mathbb S^n$ with the standard metric. Zhong--Yang's estimate \cite{zy84} asserts that $\lambda_1(M^n)\gs \pi^2/{\rm diam}^2(M^n)$ if $M^n$ has nonnegative
   Ricci curvature. The statement is optimal. In \cite{hw07}, Hang--Wang proved that if the equality holds, then $M^n$ must be  isometric to the circle of
    radius ${\rm diam}(M^n)/\pi.$ Chen--Wang in \cite{cw94,cw97} and Bakry--Qian in \cite{bq00}   put these two lower bound estimates in a same framework,
    which is the following comparison theorem:
\begin{thm}\label{thm1.1} {\rm(Chen--Wang \cite{cw94,cw97}, Bakry--Qian \cite{bq00})}\indent Let $M^n$ be a compact Riemannian manifold of dimension
 $n$ $($with a convex boundary or without boundary$)$ and $Ric(M^n)\gs (n-1)K$. Then the first non-zero $($Neumann$)$ eigenvalue satisfies
  $$\lambda_1(M^n)\gs \lambda_1(K,n,d),$$
where $d$ is the diameter of $M^n$, $\lambda_1(K,n,d)$ denotes the first non-zero Neumann eigenvalue of the following one-dimensional model:
$$v''(x)-(n-1)T(x) v'(x)=-\lambda v(x) \quad x\in (-\frac{d}{2},\frac{d}{2}),\qquad v'(-\frac{d}{2})=v'(\frac{d}{2})=0$$
and
\begin{equation*}T(x)=\begin{cases}\sqrt{K}\tan(\sqrt Kx)& {\rm if} \ K\gs0,\\ -\sqrt{-K}\tanh(\sqrt{-K}x)& {\rm if} \ K<0.\end{cases}
\end{equation*}
\end{thm}

In \cite{petu03}, Petrunin extended the  Lichnerowicz's estimate to Alexandrov spaces with curvature $\gs1.$ More generally, in \cite{lv07},
 Lott--Villani extended Lichnerowicz's estimate to a metric measure space with
  $CD(n,(n-1))$\footnote{This is a generalized notion of Ricci curvature bounded below by $n-1$ on metric measure spaces.
   We refer the reader to a survey \cite{zz10-2} for others generalizations of the lower bounds of Ricci curvature on singular spaces, in particular on Alexandrov spaces.}.
    In particular, Lichnerowicz's estimate holds on  an $n$-dimensional Alexandrov space $M$ with $Ric(M)\gs n-1$.
    This was also proved  by the second and third named authors  in \cite{zz10-3} via a different method.

 For simplicity, we always assume that the Alexandrov space $M$ has empty boundary.  Our first result in this paper is an extension of  the above
  comparison result (Theorem \ref{thm1.1}) on Alexandrov spaces. Explicitly, we will prove the following:
\begin{thm}\label{thm1.2}  Let $M$ be a compact $n$-dimensional Alexandrov space  without boundary  and $Ric(M)\gs (n-1)K$. Then its the first non-zero eigenvalue satisfies
 $$\lambda_1(M)\gs \lambda_1(K,n,d),$$
where $d$ is the diameter of $M$ and  $\lambda_1(K,n,d)$ as above in Theorem \ref{thm1.1}.
\end{thm}

As a consequence, by combining with the maximal diameter theorem in \cite{zz10-1}, we obtain an Obata type theorem (see \cite{S} for the case of orbifolds).
\begin{cor}\label{cor1.3} Let $M$ be a compact $n$-dimensional Alexandrov space  without boundary  and $Ric(M)\gs (n-1)$. If $\lambda_1(M)=n$,
then $M$ is isometric to a spherical suspension over an $(n-1)$-dimensional Alexandrov space  with curvature $\gs1.$\end{cor}

There are two different approaches  to prove Theorem \ref{thm1.1}.  One is a probabilistic  way:  Chen--Wang \cite{cw94,cw97} used the Kendall--Cranston coupling
method to prove Theorem \ref{thm1.1}. This way does not work directly on Alexandrov spaces, since it is not clear how to construct Brownian motions and how to
define SDE on Alexandrov spaces.   The other is an analytic way, given by Bakry and the first author   in \cite{bq00}. The latter approach consists of three parts.
In the first part, by combining  Bochner's formula and a smooth maximum principle argument, Kr\"oger in \cite{k92} obtained a comparison theorem for the gradient
of the eigenfunctions, which was also proved by Bakry--Qian in \cite{bq00} for   general differential operators $L$ with curvature-dimension condition $CD(n,R)$.
Secondly, by using the comparison result on the gradient of eigenfunctions  and the boundness of Hessian of eigenfunctions, Bakry--Qian proved a comparison theorem
for the maximum of eigenfunctions. In the last part, Bakry--Qian \cite{bq00} developed  a deep analysis on the one-dimensional models to prove  Theorem \ref{thm1.1}.
For Alexandrov spaces with Ricci curvature bounded below, a Bochner type formula has been established by the second and third authors in \cite{zz10-3}. Our proof of
Theorem \ref{thm1.2} is basically along the line of Bakry--Qian's proof in \cite{bq00}. However, we must overcome the difficulties bringing in due to
lacking of a smooth maximum principle and the boundedness of Hessian of eigenfunctions on Alexandrov spaces. To overcome the first difficulty,
we will replace the smooth maximum principle argument by a method of upper bound estimate for weak solutions of elliptic equations.
To overcome the second difficulty, we will appeal to a mean value inequality of Poisson equations in \cite{zz10-3}.

After we completed this paper (which was posted on Arxiv in Feb. 2011), we noted that Andrews-Clutterbbuck in \cite{ac12} provided a heat equation proof for the above Theorem \ref{thm1.1}. Recently, the sharp estimate has been also extended to Finsler manifolds by Wang-Xia in \cite{wx11}.

The second purpose of this paper is to extend Li--Yau's parabolic estimates from smooth Riemannian manifolds to Alexandrov spaces.
\begin{thm}\label{thm1.4}
 Let $M$ be a compact  $n$-dimensional Alexandrov space with nonnegative Ricci curvature and $\partial M=\varnothing$.
 Assume that $u(x,t)$ is a positive solution of heat equation $\frac{\partial}{\partial t}u=\Delta u $ on $M\times [0,\infty)$. Then we have
\begin{equation}\label{eq1.1} |\nabla\log u|^2 -\frac{\partial}{\partial t}\log u\ls \frac{n}{2t} \end{equation}for any $t>0.$
\end{thm}
Here $\Delta $ is the generator of the canonical Dirichlet form on $M$ (see Section 5 for the details).
As an application of this estimate, a sharper Harnack inequality of positive solutions of heat equation is obtained (see Corollary \ref{cor5.4}).

When $M$ is a smooth Riemannian manifold, Li--Yau in \cite{ly86} proved \eqref{eq1.1}  by Bochner's formula and smooth maximum principle.
In \cite{bl06}, Bakry--Ledoux developed an abstract method to prove \eqref{eq1.1}. They used only Bochner's formula.  We will use
Bakry--Ledoux's method to prove Theorem \ref{thm1.4}.  To apply the Bochner type formula in \cite{zz10-3}, we need to  establish
necessary regularity for positive solutions of the heat equations.

The paper is organized as follows. In Section 2, we recall some necessary
materials for Alexandrov spaces.   In Section 3, we will proved the gradient estimate for the first eigenfunction and establish a comparison result for the maximum of eigenfunctions.
In Section 4, we will prove Theorem \ref{thm1.2} and some corollaries of it. Obata type theorem and some explicit lower bound estimates will be given in this section.
In the last section, we will consider the heat equations on Alexandrov spaces. Li--Yau's parabolic estimate and a sharper Harnack estimate of positive solutions
of heat equations will be obtained in this section.

\noindent\textbf{Acknowledgements.} We would like to thank Professor Jun Ling for his interesting in the paper.  The third author is partially supported by NSFC 10831008.

\section{Preliminaries on Alexandrov spaces}

Let $(X,|\cdot\cdot|)$ be a metric space. A rectifiable curve $\gamma$ connecting two points $p,q$ is called a geodesic if its length is equal to $|pq|$ and
it has unit speed. A metric space $X$ is called a geodesic space if every pair points $p,q\in X$ can be connected by \emph{some} geodesic.

Let $k\in\R$ and $l\in\mathbb N$. Denote by $\mathbb M^l_k$ the simply connected, $l$-dimensional space form of constant sectional curvature $k$.
The model spaces are $\M^2_k$. Given three points $p,q,r$ in a geodesic space $X$, we can take a comparison triangle $\triangle \bar p\bar q\bar r$
in $\M^2_k$ such that $|\bar p\bar q|=|pq|$, $|\bar q\bar r|=|qr|$ and $|\bar r\bar p|=|rp|$. If $k>0$, we add the assumption
$|pq|+|qr|+|rp|<2\pi/\sqrt{k}$. Angles $\wa_k pqr:=\angle \bar p\bar q\bar r$ etc. are called comparison angles.

A geodesic space $X$ is called an Alexandrov space (of locally curvature bounded below) if it satisfies the following property:\\
\indent(i) it is locally compact;\\
\indent(ii)  for any point $x\in X$ there exists a neighborhood $U_x$ of $x$ and a real number $\kappa$ such that, for any four different points
$p, a, b,c$ in $U_x$, we have
 $$\wa_\kappa apb+\wa_\kappa bpc+\wa_\kappa cpa\ls 2\pi.$$
The Hausdorff dimension of an Alexandrov space is always an integer.  We refer to the seminar paper \cite{bgp92} or the text book \cite{bbi01} for the details.

Let $n\gs2$ and $M$  an $n$-dimensional  Alexandrov space and $\Omega$ be a domain in $M$. The Sobolev spaces $W^{1,p}(\Omega)$ is well defined (see, for example \cite{kms01}).
We denote  by $Lip_0(\Omega)$ the set of Lipschitz continuous functions on $\Omega$ with compact support in $\Omega.$
Spaces $W_0^{1,p}(\Omega)$ is defined by the closure of $Lip_0(\Omega)$ under  $W^{1,p}(\Omega)$-norm.
We say a function $u\in W^{1,p}_{loc}(\Omega)$ if $u\in W^{1,p}(\Omega')$ for every open subset $\Omega'\Subset\Omega.$

Denote by $\rv$ the $n$-dimensional Hausdorff measure on $M$. The canonical Dirichlet energy (form) $\mathscr E: W_0^{1,2}(\Omega)\times W_0^{1,2}(\Omega)\to \R$ is defined by
 $$\mathscr E(u,v):=\int_\Omega\ip{\nabla u}{\nabla v}d{\rm vol}\qquad {\rm for}\ u,v\in W_0^{1,2}(\Omega).$$

 Given a function $u\in W_{loc}^{1,2}(\Omega)$, a functional $\mathscr L_u$ is defined on $Lip_0(\Omega)$ by
$$\mathscr L_u(\phi):=-\int_\Omega\ip{\nabla u}{\nabla \phi}d{\rm vol},\qquad \forall \phi\in Lip_0(\Omega).$$

Let $f\in L^2(\Omega)$. If $u\in W_{loc}^{1,2}(\Omega)$ such that $\mathscr L_u$
is bounded below (or above) in the following sense that
$$\mathscr L_u(\phi)\gs\int_\Omega f\phi d{\rm vol}\qquad \Big({\rm or}\quad \mathscr L_u(\phi)\ls\int_\Omega f\phi d{\rm vol}\Big)$$
 for all nonnegative $\phi\in Lip_0(\Omega)$, then  the functional $\mathscr L_u$ is a signed Radon  measure. In this case, $u$ is said to be a sub--solution
 (super--solution, resp.) of Poisson equation
 $$\mathscr L_u=f\cdot{\rm vol}.$$

 A function $u$ is a (weak) solution of Poisson equation $\mathscr L_u=f\cdot{\rm vol}$ on $\Omega$ if it is both a sub--solution and a super--solution
 of the equation. In particular, a (weak) solution of $\mathscr L_u=0$ is called a harmonic function.

If $f,g\in W^{1,2}(\Omega)$ and $\mathscr L_g$ is a signed Radon measure, then
$$\big(f\mathscr L_g\big)(\phi)=\int_\Omega\phi f\mathscr L_g(dx)=-\int_\Omega \ip{f\nabla \phi+\phi\nabla f}{\nabla g}d\rv $$
for any $\phi\in Lip_0(\Omega).$ Hence, it is easy to check that if $f, g, fg\in W^{1,2}(\Omega)$ and $\mathscr L_f, \mathscr L_g, \mathscr L_{fg}$
 are signed Radon measures, we have
$$\mathscr L_{fg}=f\mathscr L_g+g\mathscr L_f+2\ip{\nabla f}{\nabla g}\cdot\rv$$
and if, in addition, $f$ is bounded, then we have $$\mathscr L_{\Phi(f)}=\Phi'(f)\mathscr L_f+\Phi''(f)|\nabla f|^2\cdot\rv$$for any $\Phi\in C^2(\R).$

 In \cite{zz10-1}, the second and third authors introduced a notion of ``Ricci curvature has a lower bound  $R$", denoted by $Ric\gs R$.
 On an $n$-dimensional Alexandrov space $M$, the condition $Ric\gs R$ implies that $M$ (equipped with its Hausdorff measure) satisfies
 Sturm--Lott--Villani's convature dimension condition $CD(n,R)$ \cite{s06,lv07} and  Kuwae--Shioya--Ohta's infinitesimal Bishop-Gromov condition
  (or measure contraction property) $BG(n,R)$  \cite{ks07,o07} (see \cite{petu09} and Appendix in \cite{zz10-1}). Of course,
  an $n$-dimensional Alexandrov space $M$ with curvature $\gs K$ must have $Ric(M)\gs (n-1)K$.

In \cite{zz10-3},  the following Bochner type formula was established.
\begin{thm}$($Theorem 1.2 in \cite{zz10-3}$)$\label{thm2.1}\indent Let $M$ be an $n$-dimensional Alexandrov space with Ricci curvature bounded from below by $R$,
and $\Omega$ be a bounded domain in $M$.
 Let  $F(x,s):\Omega\times[0,+\infty)\to \R$ be a Lipschitz function  and satisfy the following:\\
\indent (a)\indent  there exists a zero measure set $\mathcal N\subset \Omega$ such that for all $s\gs0$, the functions $F(\cdot,s)$ are differentiable at any
 $x\in \Omega\backslash \mathcal N;$\\
\indent (b)\indent the function $F(x,\cdot)$ is of class $C^1$ for  all $x\in \Omega$  and the function $\frac{\partial F}{\partial s}(x,s)$ is continuous,
 non-positive on $\Omega\times [0,+\infty)$.

Suppose that $u$ is  Lipschitz on $\Omega$ and satisfies
  $$\mathscr L_u=F(x,|\nabla u|^2)\cdot\rv. $$

Then we have $|\nabla u|^2\in W^{1,2}_{loc}(\Omega)$ and
 \begin{equation*}
 \mathscr L_{|\nabla u |^2}
 \gs \Big(\frac{2}{n}F^2(x,|\nabla u|^2)+2\ip{\nabla u}{\nabla F(x,|\nabla u|^2)}+2R|\nabla u|^2\Big)\cdot\rv,
\end{equation*}
provided  $|\nabla u|$ is lower semi-continuous at almost all $x\in \Omega$  $($That is, there exists a representative of $|\nabla u|$,
 which is lower semi-continuous at almost all $x\in \Omega.)$.
\end{thm}
For our purpose in this paper, we give the following corollary.
\begin{cor}\label{co2.2} Let $M$ be an $n$-dimensional Alexandrov space with Ricci curvature bounded from below by $R$,
and $\Omega$ be a bounded domain in $M$. Let $f$ be a Lipschitz continuous function in $\Omega$ and $u\in W^{1,2}_{loc}(\Omega)$ satisfies
  $$\mathscr L_u=f\cdot\rv. $$

Suppose that a real function $\Phi(t)\in C^3(\mathbb R)$ satisfies $\Phi'(t)\not\equiv0$ and $\Phi''(t)\ls0$ for all $t$ in the range of $u$.
 Then we have $|\nabla \Phi(u)|^2\in W^{1,2}_{loc}(\Omega) $ and
 \begin{equation*}
 \mathscr L_{|\nabla \Phi(u) |^2}
 \gs \Big(\frac{2}{n}\Psi^2+2\ip{\nabla \Phi(u)}{\nabla \Psi}+2R|\nabla\Phi(u)|^2\Big)\cdot\rv,
\end{equation*}
where
 $$\Psi(x):=\Phi'(u)f(x)+\Phi''(u)|\nabla u|^2.$$
\end{cor}
\begin{proof} Since $f$ is Lipschitz continuous in $\Omega$, in Corollary 5.5 in \cite{zz10-3}, it is shown that $u$ is locally Lipschitz continuous in $\Omega$,
 and in Corollary 5.8 in \cite{zz10-3}, it is shown that $|\nabla u|$ is lower semi-continuous in $\Omega$. Thus, $u$ is differential at almost everywhere in $\Omega$,
  and $|\nabla \Phi(u)|$ is lower semi-continuous in $\Omega$.

Fix any open set $\Omega'\Subset\Omega$ and define the function $F(x,s): \Omega'\times[0,+\infty)$ by
$$F(x,s):=\Phi'(u)f(x)+\frac{\Phi''(u)}{[\Phi'(u)]^2}\cdot s.$$
Since both $f$ and $u$ are locally Lipschitz continuous in $\Omega$, we have the function $F(\cdot,s)$ is  Lipschitz continuous on $\Omega'\times[0,+\infty)$.
Note that $\Phi''(t)\ls0$, it is easy check that $F(x,s)$ satisfies
the conditions (a) and (b) in Theorem 1.2.

From $\mathscr L_u=f\cdot\rv$, we get
\begin{equation*}\begin{split}
\mathscr L_{\Phi(u)}&=\Phi'(u)\mathscr L_u+\Phi''(u)|\nabla u|^2\cdot\rv\\
&=
\Psi(x)\cdot\rv=F(x,|\nabla \Phi(u)|^2)\cdot\rv.
\end{split}\end{equation*}
Now, we can apply Theorem 2.1 to conclude the desired result in this corollary.
\end{proof}

 The same trick as  in the proof of  Theorem 6 in \cite{bq00} gives an improvement of the Bochner inequality as following:
\begin{cor}\label{co2.3}Let $M$, $u$ and $f$ be as above in Corollary \ref{co2.2}, and let $\Omega\subset M$ be an open set. Assume $|\nabla u|\gs c>0$ a.e. on $\Omega$
 for some constant $c$. Then we have the following improved Bochner formula
\begin{equation}\label{e2.1}
\mathscr L_{|\nabla u|^2}\gs \Big(\frac{2}{l}f^2+2\ip{\nabla u}{\nabla f}+2R|\nabla u|^2+ \frac{2l}{l-1}\Big(\frac{f}{l}-\frac{\ip{\nabla u}{\nabla |\nabla u|^2}}{2|\nabla u|^2}\Big)^2\Big)\cdot\rv
\end{equation}on $\Omega$, for all real number $l\gs n.$
\end{cor}

The following mean value inequality  was also obtained in \cite{zz10-3}.
\begin{prop}$($Corollary 4.5 in \cite{zz10-3}$)$\label{pr2.2}\indent Let $M$ be an $n$-dimensional Alexandrov space with Ricci curvature bounded from below by
 $R$ and let $\Omega$ be a domain in $  M$. If $u$ is continuous and satisfies that $\mathscr L_u\ls c_1\cdot \rv$ on $\Omega$ and $u\gs0$, then, for any $p\in \Omega$,
  there exists a constant $c_2=c_2(n,\Omega, p, c_1)$ such that
  $$\frac{1}{\rv\big( B_o(r)\subset T_p^{R/(n-1)}\big)}\int_{B_p(r)}ud\rv\ls u(p)+ c_2 r^2$$
   for any sufficiently small $r$ with $B_p(r)\Subset\Omega$, where $T_p^{R/(n-1)}$ is  the $\frac{R}{n-1}$-cone over $\Sigma_p$, the space of directions (see \cite{bbi01} p. 354).
\end{prop}

\section{comparison theorems on gradient and maximum of eigenfunctions }
Let $M$ be a compact $n$-dimensional Alexandrov space without boundary and $Ric(M)\gs R:=(n-1)K$. Let $\lambda_1$ be the non-zero first eigenvalue and $f$
 be a first eigenfunction on $M$. That is,  $f$ is a minimizer of
 $$\lambda_1:=\inf\Big\{\frac{\int_M|\nabla \phi|^2d\rv}{\int_M\phi^2d\rv}:\ \phi\in Lip(M)\backslash\{0\}\quad {\rm and}\quad \int_M\phi d\rv=0\Big\}.$$
  It is easy to check  that $\mathscr L_f$ is a measure and satisfies \begin{equation}\label{eq3.1}\mathscr L_f=-\lambda_1 f\cdot \rv.\end{equation}

We  set $G=|\nabla f|^2$ in this section. The following regularity result is necessary for us.
\begin{lem}\label{lem3.1} $G$ is lower semi-continuous on $M$ and lies in $W^{1,2}(M).$
\end{lem}
\begin{proof}
It was proved that $f$ is Lipschitz continuous on $M$ in \cite{petu03}  (see also  Theorem 4.3 in \cite{gko10} or Corollary 5.5 in \cite{zz10-3}).
Now by applying Theorem \ref{thm2.1} to the equation \eqref{eq3.1}, we may deduce that $G\in W^{1,2}(M)$ and
$$\mathscr L_G\gs (-2\lambda_1+2R)G\cdot\rv.$$

If $2\lambda_1-2R\ls0$, then $\mathscr L_G\gs0$. This concludes that $G$ has a lower semi-continuous representation in $W^{1,2}(M)$ (see Theorem 5.1 in \cite{km02}).
 If $\mu:=2\lambda_1-2R>0$,  we consider the function $g=e^{\sqrt \mu t}G$ on $M \times \R$ with directly product metric and obtain $\mathscr L_g\gs0$.
  Hence $g$ has also a lower semi-continuous representation, and therefore $G$ is lower semi-continuous.
\end{proof}

Let us recall the one-dimensional model operators $L_{R,l}$ in \cite{bq00}.
Given $R\in \R$ and $l> 1$, the one-dimensional models $L_{R,l}$ are defined as follows: setting $K=R/(l-1)$,\\
\indent (1)\indent  If $R>0$, $L_{R,l}$ defined on $(-\pi/2\sqrt{K}, \pi/2\sqrt{K})$ by
 $$L_{R,l}v(x)=v''(x)-(l-1)\sqrt K\tan (\sqrt Kx)v'(x);$$
 \indent (2)\indent  If $R<0$, $L_{R,l}$ defined on $(-\infty, \infty)$ by
 $$L_{R,l}v(x)=v''(x)+(l-1)\sqrt{- K}\tanh (\sqrt{- K}x)v'(x); $$ and\\
\indent (3)\indent  If $R=0$, $L_{R,l}$ defined on $(-\infty, \infty)$ by
$$L_{R,l}v(x)=v''(x).$$
We refer the readers to \cite{bq00} for the  properties of $L_{R,l}$. \\

The first purpose of this section is to show the following comparison result on the gradients of eigenfunctions, which is an extension of
Kr\"oger's comparison result in \cite{k92}. In smooth case, the proof of this result in \cite{k92, bq00} relies on smooth maximum principle.
For (singular) Alexandrov spaces we need to use a method of upper bound estimate for weak solutions of elliptic equations.
\begin{thm}\label{thm3.2}Let $l\in\R$ and $l\gs n$. Suppose $\lambda_1>\max\{0,\frac{lR}{l-1}\}$.
Let $v$ be a Neumann eigenfunction of $L_{R,l}$ with respect to the same eigenvalue $\lambda_1$ on some interval. If
 $[\min f,\max f]\subset [\min v,\max v]$, then
 $$G:=|\nabla f|^2\ls (v'\circ v^{-1})^2(f).$$
\end{thm}
\begin{proof}Without loss of generality, we may assume that $$[\min f,\max f]\subset (\min v,\max v).$$

 Denote by $T(x)$ the function such that
  $$L_{R,l}(v)=v''-Tv'.$$
   As in Corollary 3 in Section 4 of \cite{bq00},
  we can choose a smooth bounded function $h_1$ on $[\min f,\max f]$ such that
   $$h'_1<\min\{Q_1(h_1), Q_2(h_1)\},$$
 where $Q_1, \ Q_2$ are given by following
\begin{align*}Q_1(h_1)&:=-(h_1-T)\Big(h_1-\frac{2l}{l-1}T+\frac{2\lambda_1v}{v'}\Big),\\
Q_2(h_1)&:=-h_1\Big(\frac{l-2}{2(l-1)}h_1-T+\frac{ \lambda_1v}{v'}\Big).
\end{align*}
We can then take a smooth function $g$ on $[\min f,\max f]$ such that $g\ls0 $ and $g'=-\frac{h_1}{v'}\circ v^{-1}.$

Now define a function $F$ on $M$ by $$\psi(f) F=G-\phi(f),$$
 where $$\psi(f):=e^{-g(f)}\qquad{\rm and}\qquad\phi(f):=(v'\circ v^{-1})^2(f).$$
  It  suffices to show $F\ls0$ on $M$.

Let us argue by contradiction.  Suppose there exists a positive small number $\epsilon_0$ such that the set $\{x\in M: F(x)\gs \epsilon_0\}$ has   positive measure.

Consider the set $\Omega=\{x\in M:\ F(x)>\frac{\epsilon_0}{2}\}$. By Lemma \ref{lem3.1} and the continuity of $f$, we know that $F$ is lower semi-continuous on $M$,
 hence, $\Omega$ is an open subset in $M$. Without loss of generality, we may assume that $\Omega$ is connected.
 Since $\psi(f)\gs1$ and $\phi(f)\gs0$, we have $G> \frac{\epsilon_0}{2}$ on $\Omega.$

In the calculation below, we write only $\Phi$ instead of $\Phi(f)$ for any function $\Phi$ on $\mathbb R$.

By applying  \eqref{e2.1} to $\mathscr L_f=-\lambda_1 f$, we have
\begin{equation}\label{eq3.3}
\mathscr L_G\gs \Big(-2\lambda_1 G+\frac{2\lambda_1^2}{l}f^2+2RG+\frac{2l}{l-1}\Big(\frac{\lambda_1 f}{l}+\frac{\ip{\nabla f}{\nabla G}}{2G}\Big)^2\Big)\rv.
\end{equation}
Noticing that $\ip{\nabla f}{\nabla G}=\psi'FG+\phi'G+\psi\ip{\nabla f}{\nabla F}$, we get
\begin{align*}\Big(\frac{\lambda_1 f}{l}+\frac{\ip{\nabla f}{\nabla G}}{2G}\Big)^2=&\Big(\frac{\lambda_1 f}{l}+\frac{\phi'}{2}\Big)^2+\frac{\psi'^2F^2}{4}+\frac{\psi^2(\ip{\nabla f}{\nabla F})^2}{4G^2}\\&+\Big(\frac{\lambda_1 f}{l}+\frac{\phi'}{2}\Big)\cdot\Big(F\psi'+\frac{\psi\ip{\nabla f}{\nabla F}}{G}\Big)+\frac{F\psi'\psi}{2G}\ip{\nabla f}{\nabla F}.
\end{align*}
Now $\phi(f)$ the following (see pp. 133 in \cite{bq00})
$$\phi(-2\lambda_1-\phi''+2R)+\lambda_1 f\phi'+\frac{2}{l}\lambda_1^2f^2+\frac{2l}{l-1}\Big(\frac{\lambda_1 f}{l}+\frac{\phi'}{2}\Big)^2=0.$$
Putting these equations to $\mathscr L_{\psi F}=\mathscr L_{G}-\mathscr L_{\phi}$, we have
$$\mathscr L_F\gs \mathcal A\cdot\rv,$$
 where
 \begin{align*}\mathcal A=&\frac{l}{2(l-1)}\frac{\psi'^2}{\psi}F^2-\frac{\psi''}{\psi}FG+\frac{2l}{l-1}\ip{\nabla f}{\nabla F}\Big(\frac{1}{G}(\frac{\lambda_1 f}{l}+\frac{\phi'}{2})+\frac{F\psi'}{2G}\Big)\\ &+\frac{1}{\psi}\Big(\lambda_1 \psi'f+\frac{2l\psi'}{l-1}(\frac{\lambda_1 f}{l}+\frac{\phi'}{2})+(-2\lambda_1+2R-\phi'')\Big)F-2\frac{\psi'}{\psi}\ip{\nabla f}{\nabla F}.
 \end{align*}
Then by substituting  $G=\psi F+\phi$ and $\psi=e^{-g} $ into the above expression, we obtain the following inequality \begin{equation}\label{eq3.2}\mathscr L_F\gs \Big(\psi(f)T_1 \cdot F^2+T_2\cdot F+T_3\ip{\nabla f}{\nabla F}\Big)\rv,
 \end{equation}
 where
 $$v'^2T_1=Q_2(h_1)-h'_1,\qquad T_2=Q_1(h_1)-h'_1$$
 and $$T_3=\frac{2l}{l-1}\Big(-\frac{g'}{2}+\frac{1}{2G}(\frac{2\lambda_1 f}{l}+\phi'+\phi g')\Big)+2g'.$$

Note that  both $T_1$ and $T_2$ are positive, and  both $T_3$ and $|\nabla f|$ are bounded on $\Omega$. It follows from \eqref{eq3.2}
 that  \begin{equation}\label{eq3.4}\mathscr L_F\gs -c|\nabla F|\cdot\rv\end{equation}on $\Omega$ for some constant $c$.

Recall that we have assumed that the set $\{x\in M:\ F(x)\gs \epsilon_0\}$ has positive measure. To get the desired contradiction,
 we only need to show \begin{equation}\label{eq3.5} \sup_\Omega F \ls \frac{\epsilon_0}{2}.\end{equation}

Take any constant $k$ to satisfy $\epsilon_0/2\ls k<\sup_\Omega F$, and set $\phi_k=(F-k)^+$. (If no such $k$ exists, we are done.)
By the definition of domain $\Omega$, we have $\phi_k\in W^{1,2}_0(\Omega).$ From \eqref{eq3.4}, we have
\begin{equation*}\begin{split}
\int_\Omega\ip{\nabla F}{\nabla \phi_k}d\rv&=-\int_\Omega \phi_kd\mathscr L_F\ls c\int_\Omega\phi_k|\nabla F| d\rv\\
&\ls c\Big(\int_{\Omega_k}|\nabla F|^2d\rv\Big)^{1/2}\Big(\int_{\Omega_k}\phi_k^2 d\rv\Big)^{1/2},
\end{split}\end{equation*}
where $\Omega_k={\rm supp}|\nabla\phi_k|\subset{\rm supp}\phi_k\subset \Omega.$
Here we have used the fact that $|\nabla \phi_k|=|\nabla F| $ in ${\rm supp}\phi_k.$ Hence, we have
 \begin{equation}\label{eq3.6}
 \int_{\Omega_k}|\nabla \phi_k|^2d\rv\ls c^2\int_{\Omega_k}  \phi_k^2d\rv.
 \end{equation}
Since $\Omega$ is bounded (by  that $M$ is compact) and $Ric(M)\gs R$, we have the following Sobolev inequality on $\Omega$
(see, for example \cite{kms01} or \cite{zz10-2}):
there exists $\nu>2$ and $C_S=C_S(n,\nu,\Omega)>0$  such that
\begin{equation}\label{sobo}
C_S\Big(\int_\Omega |\psi|^{\nu}d\rv\Big)^{2/\nu}\ls \int_\Omega |\nabla\psi|^2d\rv  ,\qquad \forall \psi\in W^{1,2}_0(\Omega).
\end{equation}
By  combining with \eqref{eq3.6}, we get
\begin{align*}
\|\phi_k\|^2_{L^2(\Omega_k)}&\ls \|\phi_k\|^2_{L^\nu(\Omega_k)}\cdot(\rv(\Omega_k))^{1-2/\nu}\\&\ls C_S^{-1}\int_\Omega |\nabla\phi_k|^2d\rv\cdot(\rv(\Omega_k))^{1-2/\nu}\\
&\ls c^2C_S^{-1}\cdot(\rv(\Omega_k))^{1-2/\nu}\cdot\|\phi_k\|^2_{L^2(\Omega_k)}.
\end{align*}
Thus we deduce that
 $$\rv(\Omega_k)\gs C$$
for some constant $C=C(c,n,\nu,C_S)>0$, which is independent of $k$.
Noting that $\Omega_k\subset {\rm supp}|\nabla F|\cap\{F\gs k\}$ and letting $k$ tend to $\sup_\Omega F$, we have
 $$\rv({\rm supp}|\nabla F|\cap\{F=\sup_\Omega F\})\gs C.$$
This is impossible, since $|\nabla F|=0$ a.e. in $\{F=\sup_\Omega F\}$ (see Proposition 2.22 in \cite{c99}).
Hence the desired   \eqref{eq3.5} is proved. Therefore, we have completed the proof of Theorem \ref{thm3.2}.
\end{proof}

Given $R, l\in \R$ with $l\gs n$ and $\lambda_1>\max\{\frac{lR}{l-1},0\},$ let  $v_{R,l}$ be the solution of the equation
$$L_{R,l}v=-\lambda_1 v$$
with initial value $v(a)=-1$ and $ v'(a)=0$, where
$$a=\begin{cases}-\frac{\pi}{2\sqrt{R/(l-1)}}&\quad {\rm if}\quad R>0,\\0 &\quad {\rm if }\quad  R\ls0.\end{cases}$$
We denote
$$b=\inf\{x>a: \ v'_{R,l}(x)=0\}$$ and  $$m_{R,l}=v_{R,l}(b).$$

The second purpose of this section is to show the following comparison result on the maximum of eigenfunctions.
\begin{thm}\label{th3.3}   Let $M$ be $n$-dimensional Alexandrov space without boundary and $Ric(M)\gs R$.
Suppose that $M$ has the first eigenvalue $\lambda_1$ and  a corresponding eigenfunction $f$.
Suppose $\lambda_1>\max\{0,\frac{nR}{n-1}\}$ and $\min f=-1,$ $\max f\ls 1.$
 Then we have
$$\max f\gs m_{R,n}.$$
\end{thm}

In smooth case, the proof of this result in \cite{bq00} relies on the fact that the Hessian of $f$ is bounded. Since we are not
sure the existence of the Hessian for an eigenfunction on Alexandrov spaces, we have to give an alternative argument.
\begin{proof}[Proof of Theorem \ref{th3.3}] Let us argue by contradiction. Suppose  $\max f< m_{R,n}.$

Since $m_{R,l}$ is continuous on $l$, we can find some real number $l>n$ such that
$$\max f\ls m_{R,l}\quad {\rm and}\quad \lambda_1>\max\{0,\frac{lR}{l-1}\}.$$
Denote $v=v_{R,l}$. Recall from (the same  proof of) Proposition 5 of \cite{bq00} that the ratio $$R(s)=-\frac{\int_Mf 1_{\{f\ls v(s)\}}d\rv}{\rho(s)v'(s)}$$ is increasing on $[a,v^{-1}(0)]$ and decreasing on $[v^{-1}(0), b]$, where the function $\rho$ is
\begin{equation*}
\rho(s):=\begin{cases} \cos^{l-1}(\sqrt Ks)&\quad {\rm if}\quad K=R/(l-1)>0;\\
s^{l-1}&\quad {\rm if}\quad K=R/(l-1)=0;\\ \sinh^{l-1}(\sqrt{- K}s)&\quad {\rm if}\quad K=R/(l-1)<0.\end{cases}
\end{equation*}

 It follows that for any $s\in[a,v^{-1}(-1/2)]$, we have
\begin{equation}\label{eq3.8}\rv(\{f\ls v(s)\})\ls -2\int_M f1_{\{f\ls v(s)\}}d\rv\ls 2C\rho(s)v'(s),\end{equation}
where $C=R(v^{-1}(0)).$

Take $p\in M$ with $f(p)=-1$. By
 $$f-f(p)\gs0,\quad {\rm and} \quad\mathscr L_{f-f(p)}=-\lambda_1 f\cdot\rv\ls \lambda_1 \cdot\rv.$$
  The mean value inequality, Proposition \ref{pr2.2}, implies that  there exists a constant $C_1$ such that
$$\frac{1}{\rv\big( B_o(r)\subset T_p^{R/(n-1)}\big)}\int_{B_p(r)}(f-f(p))d\rv\ls C_1 r^2$$
for any sufficiently small $r>0$.
Let  $A(r)=\{f-f(p)>2C_1r^2\}\cap B_p(r)$. Then
$$\frac{\rv(A(r))}{\rv(B_p(r))}\ls \frac{\int_{B_p(r)}(f-f(p))d\rv}{2C_1r^2\rv(B_p(r))}\ls \frac{\rv\big( B_o(r)\subset T_p^{R/(n-1)}\big)}{2\rv(B_p(r))}\ls \frac{2}{3}$$
for any sufficiently small $r>0.$ Here we have used the fact
$$\lim_{r\to0^+}\frac{\rv\big( B_o(r)\subset T_p^{R/(n-1)}\big)}{\rv(B_p(r))} =1.$$
Hence
\begin{equation*}\begin{split}\frac{1}{3}\rv(B_p(r))&\ls \rv\big(B_p(r)\backslash A(r)\big)\ls \rv(\{f\ls f(p)+2C_1r^2\})\\&=\rv( \{f\ls -1+2C_1r^2\})
\end{split}\end{equation*}
for any sufficiently small $r>0.$
By combining this with \eqref{eq3.8}, we have
 \begin{equation}\label{eq3.9}\rv(B_p(r))\ls 6C\rho(s)\cdot v'(s)\end{equation}
 for any sufficiently small $r>0$, where
  $$s=v^{-1}\big(-1+2C_1r^2\big)=v^{-1}\big(v(a)+2C_1r^2\big).$$
Rewriting $L_{R,l}v=-\lambda_1v$ as $(\rho v')'=-\lambda_1\rho v$ and noting that $v'(a)=0$,
 we get
 $$(\rho v')(s)=-\lambda_1\int_a^s\rho vdt\quad{\rm and}\quad \frac{v'(s)-v'(a)}{s-a}=-\lambda_1\frac{\int_a^s\rho vdt}{(s-a)\rho(s)}.$$
By applying L'Hospital's rule, we have
 $$v''(a)=-\lambda_1\lim_{s\to a}\frac{\int_a^s\rho vdt}{(s-a)\rho(s)}=-\lambda_1\lim_{s\to a}\frac{ v(s)}{1+(s-a)\rho'(s)/\rho(s)}.$$
Noting that
$$v(a)=-1\qquad {\rm and}\quad \lim_{s\to a}(s-a)\frac{\rho'(s)}{\rho(s)}=l-1,$$
we  get $v''(a)=\lambda_1/l$.
Hence there exists two  constants $C_2$ and $C_3$ such that
 $$0<C_2\ls v''\ls C_3<\infty$$
 in a neighborhood of $a.$ The combination of $v'(a)=0$ and $v''(s)\gs C_2$ implies  that
 \begin{equation}\label{eq3.10}v(s)-v(a)\gs \frac{C_2}{2}(s-a)^2\end{equation}
 for $s$ sufficiently near $a$. On the other hand, the combination of $v'(a)=0$ and
  $0<v''(s)\ls C_3$ implies  that
   \begin{equation*}\label{eq3.11}0\ls v'(s)\ls C_3(s-a)\end{equation*}
  for $s$ sufficiently near $a$. Note that, by the definition of function $\rho$,
  \begin{equation*}\lim_{s\to a^+}\frac{\rho^{\frac{1}{l-1}}(s)}{s-a}=\begin{cases}\sqrt K,& \qquad {\rm if}\quad K=R/(l-1)>0;\\ 1, & \qquad {\rm if}\quad K=R/(l-1)=0;
  \\ \sqrt{-K},&\qquad {\rm if}\quad K=R/(l-1)<0.
  \end{cases}\end{equation*}
  Thus, we have
  $$\rho(s)\ls C'_3(s-a)^{l-1}$$
  for $s$ sufficiently near $a$ and for some constant $C'_3$.
  By combining with $0\ls v'(s)\ls C_3(s-a)$, 
  we have
\begin{equation}\label{eq3.12} \rho(s)v'(s)\ls C_3\cdot C'_3\cdot(s-a)^l:=C_4(s-a)^l\end{equation}
for $s$ sufficiently near $a$.

 The combination of \eqref{eq3.9}, \eqref{eq3.12} and \eqref{eq3.10} implies that
 $$\rv(B_p(r))\ls 6C\cdot C_4(s-a)^l\ls6C\cdot C_4\cdot\Big(\frac{2}{C_2}\big(v(s)-v(a)\big)\Big)^{l/2}$$
 for $s$ sufficiently near $a$.
 Noting that $v(s)-v(a)=2C_1r^2$, we have
  \begin{equation}\label{eq3.13}\rv (B_p(r))\ls 6C\cdot C_4\cdot\Big(\frac{4C_1}{C_2}r^2\Big)^{l/2}:= C_5r^l\end{equation}
for any  sufficiently small $r$.

Fix $r_0>0$. By Bishop--Gromov volume comparison, we have
$$\frac{\rv(B_p(r))}{\mathcal H^n(B^K(r))}\gs \frac{\rv(B_p(r_0))}{\mathcal H^n(B^K(r_0))}$$
for any $0<r<r_0$, where $\mathcal H^n(B^K(r))$ is the volume of a geodesic ball with radius $r$ is $n$-dimensional simply connected space form
with sectional curvature $K$. Thus, there exists a constant $C_6$ such that
\begin{equation}\label{eq3.14}\rv (B_p(r))\gs\frac{\rv(B_p(r_0))}{\mathcal H^n(B^K(r_0))}\cdot\mathcal H^n(B^K(r))\gs C_6 r^n\end{equation}
 for any sufficiently small $r.$

  The combination of \eqref{eq3.13} and \eqref{eq3.14} implies that $C_5\cdot r^{l-n}\gs C_6$ holds for any sufficiently small $r$. Hence, we get $l\ls n$. This contradicts to the assumption  $l>n.$ Therefore, the proof of Theorem \ref{th3.3} is finished.
\end{proof}

\section{comparison theorems on the first eigenvalue and its applications}
In this section, we will prove Theorem \ref{thm1.2} in Introduction and its corollaries.
\begin{proof}[Proof of Theorem \ref{thm1.2}]Without loss of generality, we may assume $K\in\{-1,0, 1\}.$
Let $\lambda_1$ and $f$ be the first non-zero eigenvalue and a corresponding eigenfunction with $\min f=-1$ and $\max f\ls1.$

By Lichnerowicz's estimate, we have $\lambda_1\gs n$, if $K=1.$  Now fix any $R<(n-1)K$, we have $$\lambda_1>\max \{\frac{nR}{n-1},0\}.$$
Then, by   using the above Theorem \ref{th3.3} and Corollary 1 and 2 in Section 3 of \cite{bq00}, we can find an interval $[a,b]$ such that the one dimensional model operator $L_{R,n}$ has the first Neumann eigenvalue $\lambda_1$ and a corresponding eigenfunction $v$ with $\min v=-1$, $\max v=\max f$. Applying Theorem 13 in Section 7 of \cite{bq00}, we have \begin{equation}\label{eq4.1}\lambda_1\gs \lambda_1\big(R/(n-1),n,b-a\big),\end{equation} where $\lambda_1\big(R/(n-1),n,b-a\big)$ is the first non-zero Neumann eigenvalue of $L_{R,n}$ on the symmetric interval $(-\frac{b-a}{2},\frac{b-a}{2}).$

By  Theorem \ref{thm3.2}, we have $$|\nabla (v^{-1}\circ f)|\ls1.$$The canonical Dirichlet form $\mathscr E$ induces a pseudo-metric $$d_\mathscr E(x,y):=\sup\{u(x)-u(y):\ u\in W^{1,2}(M)\cap C(M)\ {\rm and}\ |\nabla u|\ls1 \ a.e.\}.$$
Since $f$ is Lipschitz continuous and $|\nabla (v^{-1}\circ f)|\ls1,$ we have $$b-a=v^{-1}(\max f)-v^{-1}(\min f)\ls \max_{x,y\in M}d_\mathscr E(x,y).$$ On the other hand, Kuwae--Machigashira--Shioya  in \cite{kms01} proved that the induced pseudo-metric $d_\mathscr E(x,y)$ is equal to the origin metric $d(x,y)$. Then $\max_{x,y\in M}d_\mathscr E(x,y)$ is equal to  $d$, the diameter of $M$. By combining this with \eqref{eq4.1}, we have
\begin{equation*}\lambda_1\gs \lambda_1\big(R/(n-1),n,d\big).\end{equation*}
Therefore, Theorem \ref{thm1.2} follows from the combination of this and the arbitrariness of $R.$
\end{proof}

In the rest of this section, we will apply Theorem \ref{thm1.2} to conclude some explicit lower bounds for $\lambda_1$.

 The same computation as in \cite{cw97} gives the following  explicit lower bounds for $\lambda_1(M)$:
\begin{cor}\label{cor4.1}  {\rm(Chen--Wang \cite{cw97})}\indent Let $M$ be a compact $n (\gs 2)$-dimensional Alexandrov space without boundary  and $Ric(M)\gs (n-1)K$. Then its first non-zero eigenvalue $\lambda_1(M)$ satisfies:\\
\indent (1)\indent if $K=1$, then $$\lambda_1(M)\gs \frac{n}{1-\cos^n(d/2)} \qquad {\rm and }\qquad \lambda_1(M)\gs \frac{\pi^2}{d^2}+(n-1)\cdot\max\{\frac{\pi}{4n},1-\frac{2}{\pi}\};$$
\indent (2)\indent if $K=-1$, then $$\lambda_1(M)\gs \frac{\pi^2}{d^2}\cosh^{1-n}(\frac{d}{2})\cdot\sqrt{1+\frac{2(n-1)d^2}{\pi^4}} \qquad {\rm and }\qquad \lambda_1(M)\gs \frac{\pi^2}{d^2}-(n-1)(\frac{\pi}{2}-1);$$
\indent (3)\indent if $K=0$, then $\lambda_1(M)\gs \frac{\pi^2}{d^2},$ $\quad($Zhong--Yang's estimate \cite{zy84}$)$\\
where $d$ is the diameter of $M$.
\end{cor}
A direct computation shows that the first eigenvalue of any $n$-dimensional spherical suspension is exactly $n$. On the other hand,  by combining Corollary \ref{cor4.1} (1) and the maximal diameter theorem for Alexandrov space in \cite{zz10-2}, we  conclude the following Obata type theorem:
 \begin{cor}\label{cor4.2} Let $M$ be a compact $n$-dimensional Alexandrov space  without boundary  and $Ric(M)\gs (n-1)$. Then $\lambda_1(M)=n$ if and only if $M$ is isometric to a spherical suspension over an $(n-1)$-dimensional Alexandrov space  with curvature $\gs1.$\end{cor}


In the end of this section, we give some explicit lower bounds of $\lambda_1(M).$
\begin{cor}\label{cor4.3}   Let $M$ be a compact $n (\gs 2)$-dimensional Alexandrov space without boundary  and $Ric(M)\gs (n-1)K$. Then its  first non-zero eigenvalue $\lambda_1(M)$ satisfies
$$\lambda_1(M)\gs 4s(1-s)\frac{\pi^2}{d^2}+s(n-1)K $$for all $s\in (0,1)$,
where $d$ is the diameter of $M$.
\end{cor}
\begin{proof}
When $K>0$, we can assume $d<\pi/\sqrt{K}.$ Otherwise, $$\lambda_1(M)=nK=K+(n-1) K\gs 4s(1-s)\frac{\pi^2}{d^2}+s(n-1)K$$for all $s\in(0,1)$.

Denote by $$D=\frac{d}{2},\qquad f=v'\qquad {\rm and}\qquad  F=-(n-1)T,$$ where $v$ and $T$ are the one variable functions in Theorem \ref{thm1.1}. Clearly, we have
$$-f''=F'f+Ff'+\lambda_1(K,n,d)\cdot f$$ and $f(\pm D)=0,$ $f(x)>0$ on $x\in(-D,D).$

For any $a>1$,  by multiplying  $f^{a-1}$ and integrating over $(-D,D)$, we get
\begin{equation}\label{eq4.2}-\int^D_{-D}f^{a-1}f''dx=\int_{-D}^D(\lambda_1(K,n,d)+F')f^adx+\int_{-D}^DFf^{a-1}f'dx.\end{equation}
Next, by $f(\pm D)=0$, we have $$-\int^D_{-D}f^{a-1}f''dx=(a-1)\int^D_{-D}f^{a-2}f'^2dx=\frac{4(a-1)}{a^2}\int^D_{-D}[(f^{a/2})']^2dx.$$
On the other hand, by $f(\pm D)=0$ again, we have\begin{align*}\int_{-D}^DFf^{a-1}f'dx&=-\int_{-D}^Df\big(F'f^{a-1}+(a-1)Ff^{a-2}f'\big)dx\\
&=-\int_{-D}^D  F'f^{a}dx-(a-1)\int_{-D}^DFf^{a-1}f'dx.\end{align*}
Hence, we have $$\int_{-D}^DFf^{a-1}f'dx=-\frac{1}{a}\int_{-D}^D  F'f^{a}dx.$$
Putting these equations to \eqref{eq4.2}, we have
$$\frac{4(a-1)}{a^2}\int^D_{-D}[(f^{a/2})']^2dx=\int_{-D}^D\big(\lambda_1(K,n,d)+(1-\frac{1}{a})F'\big)f^adx.$$
Letting $s=1-\frac{1}{a}\in(0,1)$, we get
\begin{align*}4s(1-s) \int^D_{-D}[(f^{a/2})']^2dx&=\int_{-D}^D\big(\lambda_1(K,n,d)+sF'\big)f^adx\\&\ls\big(\lambda_1(K,n,d)+s\max_{x\in(-D,D)} F'\big)\int_{-D}^D f^adx.\end{align*}
Since $f^{a/2}(\pm D)=0$, by Wirtinger's inequality, we have $$4s(1-s)\Big(\frac{\pi}{2D}\Big)^2\ls \lambda_1(K,n,d)+s\max_{x\in(-D,D)} F'.$$
Note also that $$ \max_{x\in(-D,D)} F'=-(n-1)\min_{x\in(-D,D)} T'=-(n-1)K.$$
Therefore, by applying Theorem \ref{thm1.2}, we get the desired estimate.
\end{proof}
\begin{rem}
\indent(1)\indent If let $s=\frac{1}{2}$, we get $$\lambda_1(M)\gs \frac{\pi^2}{d^2}+\frac{1}{2}(n-1)K.$$
This improves  Chen--Wang's result in both $K>0$ and $K<0$. It also improves Ling's recent results in \cite{ling07}. \\
\indent(2)\indent If $K>0$, Peter Li conjectures that $\lambda_1(M)\gs \pi^2/d^2+(n-1)K$. Corollary \ref{cor4.3} implies that $\lambda_1(M)\gs \frac{3}{4}\big(\pi^2/d^2+(n-1)K\big).$\\
\indent(3)\indent If $n\ls 5$ and $K>0$, by choosing some suitable constant $s$, we have $$\lambda_1(M)\gs \frac{\pi^2}{d^2}+\frac{1}{2}(n-1)K+\frac{(n-1)^2K^2d^2}{16\pi^2}.$$
\end{rem}

\section{Li--Yau's parabolic estimates}
In this section, we consider heat equations on Alexandrov spaces.

Let $M$ be an $n$-dimensional compact Alexandrov space without boundary and let  $\mathscr E: W^{1,2}(M)\times W^{1,2}(M)\to \R$ be  the canonical Dirichlet energy.   Associated with the Dirichlet form $(\mathscr E, W^{1,2}(M))$, there exists an infinitesimal generator $\Delta$ which acts on a dense subspace $\mathbf D(\Delta)$ of $W^{1,2} (M)$, defined by
$$\int_Mg\Delta fd\rv=-\mathscr E(f,g),\qquad \forall \, f\in \mathbf D(\Delta)\ {\rm and}\ g\in W^{1,2}(M).$$
By the definition, it is easy to check that  $f\in \mathbf D(\Delta)$ implies $ \mathscr L_f=\Delta f\cdot \rv$.

By   the general theory of analytic semigroups (see for example \cite{en99}), the operator $\Delta$ generates an analytic semigroup $(T_t)_{t\gs0}$  on $ L^2(M).$
For any $f \in  L^2(M)$, $u(x,t):=T_tf(x)$ solves the (linear)  heat equation $$\frac{\partial}{\partial t}u(x,t) = \Delta u(x,t) $$
with initial value $u(x,0)=f(x)$ in the sense that\\
 \indent (1)\indent$T_tf\to f $ in $L^2(M)$, as $t\to0$;\\
 \indent (2)\indent $T_tf\in \mathbf D(\Delta)$ and $\frac{\partial}{\partial t}T_tf=\Delta T_tf $ for all $t>0$.\\
Moreover, $T_tf$ satisfies the following properties (see, for example \cite{en99})\\
\indent (3)\indent $T_tf\in D((\Delta)^m),$ for all $t > 0\ {\rm and\ all }\  m\in\mathbb N.$
If $f\in D((\Delta)^m)$, then $$(\Delta)^mT_tf=T_t(\Delta)^mf,\qquad \forall t\gs0.$$

The existence of the heat kernel was proved in \cite{kms01}. More precisely,  there exists a unique, measurable, nonnegative,
and locally H\"older continuous function  $p_t(x, y)$ on $(0,\infty)\times M \times M$ satisfying the following properties (i)--(iii):\\
\indent (i)\indent For any $f\in L^2(M)$, $x\in M$ and $t>0$, \begin{equation}\label{eq5.1}T_tf(x)=\int_M  p_t(x,y)f(y)d\rv(y);\end{equation}
\indent (ii)\indent For any $s, t > 0$ and $x, y \in M$, we have $$p_t(x,y)=p_t(y,x)>0,$$ $$p_{t+s}(x,y)=\int_M p_t(x,z) p_s(z,y)d\rv(z),$$
$$\int_M p_t(x,y) d\rv(y)=1,$$the last equality
follows from the fact that $M$ is compact;\\
\indent (iii)\indent  Denote by $0<\lambda_1\ls \lambda_2\ls \cdots$ all
the non-zero eigenvalues of $\Delta$  with multiplicity and by $\{\phi_j\}_{j=1}^\infty$ the sequence
of associated eigenfunctions which is a complete orthonormal basis of
$W^{1,2}(M)$ and $\|\phi_j\|_{L^2(M)}=1$ for all $j\in\mathbb N$. Then we have
\begin{equation}\label{eq5.2}p_t(x,y)=\frac{1}{\rv(M)} +\sum^\infty_{j=1}e^{-\lambda_jt}\phi_j(x)\phi_j(y)\end{equation}for all $t>0$ and $x,y\in  M$. Moreover, Cheng--Li proved that Sobolev inequality \eqref{sobo} implies that $$\lambda_j\gs C\cdot (j+1)^{\frac{\nu-2}{\nu}}$$for some constant $C=C(n,\nu,M, C_S)>0$ (see, for example, Section 3.5 in book \cite{sy94}).

\begin{lem}\label{lem5.1}Let $M$ be  an $n$-dimensional Alexandrov space  with $\partial M=\varnothing$. Then for any $ f\in L^2(M)$, $t>0$ and $m\in \mathbb N$, $(\Delta)^mT_tf$ is Lipschitz continuous in $M$.
\end{lem}
\begin{proof}For any  $m\in \mathbb N$,  we have $T_tf\in D((\Delta)^m) $ for all $t > 0.$ This concludes that $(\Delta)^mT_tf$ is a solution of equation
$$\frac{\partial}{\partial t}u(x,t)=\Delta u(x,t).$$

On the other hand,  the same proof of the locally Lipschitz continuity of Dirichlet heat kernel on a bounded domain (Theorem 5.14 in  \cite{zz10-2}) proves that $p_t(\cdot,y)$ is Lipschitz continuous on $M$, for any $y\in M$. From \eqref{eq5.1}, we get that any solution of equation $\frac{\partial}{\partial t}u(x,t)=\Delta u(x,t)$ is  Lipschitz continuous on $M$ (or see Theorem 4.4 in \cite{gko10}). Therefore, the proof of the lemma is completed.
\end{proof}

Next, we adapt Bakry--Ledoux's method in \cite{bl06} to prove Theorem \ref{thm1.4}.
\begin{proof}[Proof of Theorem \ref{thm1.4}]Fix any small $\epsilon>0$ and set  $f=u(\cdot,\epsilon)$. Since $u$ is positive and continuous on $M$, we know that $f$ is bounded by a positive constant from below. Without loss of generality, we may assume that $f>1$.

The combination of equation \eqref{eq5.1} and
 \begin{equation*} \int_M p_t(x,y) d\rv(y)=1  \end{equation*}
 implies $T_tf>1$ for all $t>0$. By the definition of $\mathbf D(\Delta)$, \eqref{eq5.1} and $f\gs1, \ T_tf\gs1$, we have  $\log f, \log T_tf\in \mathbf D(\Delta)$ for all $t>0.$

Fix $t>0$ and, as in \cite{bl06}, let us consider the function $$\psi(s)=T_s\Big(T_{t-s}f\cdot|\nabla \log T_{t-s}f|^2\Big), \qquad 0\ls s\ls t.$$
Setting $g_s =\log T_{t-s} f$, by $\frac{\partial}{\partial s}T_{t-s}f=-\Delta T_{t-s}f$, we have  $$\psi(s,x)=\int_M p_s(x,y)\cdot\exp g_s(y)\cdot|\nabla g_s(y)|^2dy,$$ $$\exp g_s\cdot\frac{\partial}{\partial s}g_s=-\Delta \exp g_s$$  and
\begin{equation}\label{eq5.3}\begin{split}\frac{\partial}{\partial s}\psi(s,x) &=\int_M \frac{\partial}{\partial s}p_s(x,y)\cdot\exp g_s(y)\cdot|\nabla g_s(y)|^2dy\\
 &\qquad +\int_M p_s(x,y)\cdot \frac{\partial}{\partial s}g_s(y)\cdot\exp g_s(y)\cdot|\nabla g_s(y)|^2dy\\
 &\qquad +\int_M p_s(x,y)\cdot\exp g_s(y)\cdot2\ip{\frac{\partial}{\partial s}\nabla g_s(y)}{\nabla g_s(y)}dy\\
 &=\int_M \Delta p_s(x,y)\cdot\exp g_s(y)\cdot|\nabla g_s(y)|^2dy\\&\qquad-
 T_s\Big(\Delta  \exp g_s \cdot|\nabla g_s |^2\Big)-T_s\Big(\exp g_s \cdot2\ip{\nabla (\Delta g_s +|\nabla g_s |^2)}{\nabla g_s }\Big).\end{split}
\end{equation}
By the definition of $\Delta$ and the functional $\mathscr L$, and $\exp g\in \mathbf D(\Delta)$, we have
\begin{align*}\int_M\Delta p_s(x,y)&\cdot\exp g_s(y)\cdot|\nabla g_s(y)|^2dy\\
&=\int_M p_s(x,y) d\mathscr L_{(\exp g_s(y)\cdot|\nabla g_s(y)|^2)}\\
&=T_s\Big(\Delta  \exp g_s \cdot|\nabla g_s |^2\Big)+2T_s(\ip{\nabla \exp g_s}{\nabla |\nabla g_s|^2})\\
&\quad+\int_M p_s(x,y)\exp g_s(y)d\mathscr L_{|\nabla g_s|^2}.
\end{align*}
From Lemma \ref{lem5.1} and $T_{t-s}f\gs1$, we get that  $\Delta T_{t-s}f$ is Lipschitz continuous on $M$, for all $0<s<t$. Note that   $g_s\in \mathbf D(\Delta)$.   Now, because $M$ has nonnegative Ricci curvature, we can apply Corollary 2.2 (Bochner type formula) to equation
$$\mathscr L_{T_{t-s}f}=\Delta T_{t-s}f\cdot\rv$$ and
 function $\Phi(t)=\log t$ to conclude that
\begin{equation}\label{eq5.4}\mathscr L_{|\nabla g_s|^2}\gs \Big(\frac{2(\Delta g_s)^2}{n}+2\ip{\nabla g_s}{\nabla \Delta g_s} \Big)\cdot \rv.\end{equation}
Putting these  above  equations and   the nonnegativity of $p_t(x,y)$ to \eqref{eq5.3}, we have \begin{equation}\label{eq5.5}\frac{\partial}{\partial s}\psi(s,x)\gs \frac{2}{n}T_s\Big(\exp g_s\cdot(\Delta g_s)^2\Big).
\end{equation}

 The rest of the proof follows exactly from the corresponding argument in \cite{bl06}.

From \eqref{eq5.5}, we have $\psi(0)\ls\psi(t)$, i.e.,
$$T_tf\cdot|\nabla\log T_tf|^2\ls T_t\big(f|\nabla \log f|^2\big),\qquad \forall t>0.$$
Since $T_tf\in \mathbf D(\Delta)$ for all $t>0$, the above inequality  implies   \begin{equation} \label{eq5.6}T_tf\cdot \Delta(\log T_tf)\gs T_t\big(f\Delta(\log f)\big),\qquad \forall t>0 .\end{equation}

By applying
$$\Delta g_s=\frac{\Delta T_{t-s}f}{T_{t-s}f}-|\nabla g_s|^2$$
 and Cauchy--Schwarz inequality, we have
\begin{equation*}\begin{split}
T_s(\exp g_s(\Delta g_s)^2)&=T_s\Big(T_{t-s}f\Big(\frac{\Delta T_{t-s}f}{T_{t-s}f}-|\nabla g_s|^2\Big)^2\Big)\\
&=T_s\Big(\frac{\big(\Delta T_{t-s}f-T_{t-s}f|\nabla g_s|^2\big)^2}{T_{t-s}f}\Big)\\
&\gs \Big(T_s\big(\Delta T_{t-s}f-T_{t-s}f|\nabla g_s|^2\big)\Big)^2\Big/T_s(T_{t-s}f)\\
&=(\Delta T_tf-\psi(s))^2/T_tf.
\end{split}\end{equation*}
Putting this into the equation \eqref{eq5.5}, we get
  $$\big(\psi (s)-\Delta T_tf\big)'\gs \frac{2}{nT_tf}\big(\psi (s)-\Delta T_tf\big)^2.$$
 This implies \begin{equation}\label{eq5.7}\varphi(s_2)-\varphi(s_1)\gs \frac{2}{nT_tf}(s_2-s_1)\cdot\varphi(s_2)\cdot\varphi(s_1),\end{equation} for all $0\ls s_1<s_2\ls t,$ where
 $$\varphi(s)=\psi (s)-\Delta T_tf.$$
 In particular, we have $$\varphi(t)-\varphi(0)\gs\frac{2t}{nT_tf}\varphi(t)\varphi(0).$$
 That is,
\begin{equation}\label{eq5.8}
-\varphi(0)\gs\frac{2t}{nT_tf}\varphi(t)\varphi(0)-\varphi(t)=-\varphi(t)\cdot\Big(1-\frac{2t}{nT_tf}\varphi(0)\Big).\end{equation}
Note that
$$\varphi(0)=-T_tf\Delta(\log T_tf), \qquad \varphi(t)=-T_t(f\Delta(\log  f)).$$
Hence, by the equation \eqref{eq5.8}, we have
\begin{equation}\label{eq5.9}T_tf\cdot \Delta(\log T_tf)\gs T_t\big(f\Delta(\log f)\big)\Big(1+\frac{2t}{n}\Delta(\log  T_tf)\Big)
\end{equation}for all $t>0.$

We now claim that
 \begin{equation}\label{eq5.10}1+\frac{2t}{n}\Delta(\log  T_tf)\gs 0,\qquad \forall t>0.
 \end{equation}
 Fix any $t>0$. Indeed, if $\Delta(\log  T_tf)\gs0$, we are done. Then we may assume that $\Delta(\log  T_tf)<0$. Since $T_tf\gs1>0$, the equation \eqref{eq5.6}
 implies
 $$T_t(f\Delta(\log f))<0.$$
  Now the equation \eqref{eq5.9} shows that
  $$1+\frac{2t}{n}\Delta(\log  T_tf)\gs\frac{T_tf\cdot \Delta(\log T_tf)}{T_t\big(f\Delta(\log f)\big)}\gs 0.$$
This proves the equation \eqref{eq5.10}.

 Rewriting the equation \eqref{eq5.10}, we get
$$|\nabla\log  T_tf|^2 -\frac{\partial}{\partial t}\log T_tf\ls \frac{n}{2t}\qquad \forall t>0.$$
Noting that $T_tf(x)=u(x,t+\epsilon)$, we get $$|\nabla\log  u|^2 -\frac{\partial}{\partial t}\log u\ls \frac{n}{2(t-\epsilon)}\qquad \forall t>\epsilon.$$The desired inequality \eqref{eq1.1} follows from the arbitrariness of $\epsilon$. Therefore,  the proof of Theorem \ref{thm1.4} is completed.
\end{proof}
Since $u(x,t)$ is continuous in $M\times(0,\infty)$, a direct application of Theorem \ref{thm1.4}, as in \cite{ly86}, gives the following parabolic Harnack inequality.
\begin{cor}\label{cor5.4}Let $M$ be a compact  $n$-dimensional Alexandrov space with nonnegative Ricci curvature and $\partial M=\varnothing$. Assume that $u(x,t)$ is a positive solution of heat equation $\frac{\partial}{\partial t}u=\Delta u $ on $M\times [0,\infty)$. Then we have
\begin{equation*}   u(x_1,t_1)\ls  u(x_2,t_2)\Big(\frac{t_2}{t_1}\Big)^{\frac{n}{2}}\exp\Big(\frac{|x_1x_2|^2}{4(t_2-t_1)}\Big) \end{equation*}for all $x_1,x_2\in M$ and  $ 0<t_1<t_2<\infty.$
\end{cor}
\noindent The Harnack inequality is sharper than that in \cite{s96}.

\end{document}